\documentclass[12pt]{article}
\usepackage[english]{babel}
\usepackage{subfig}
\usepackage{hyperref,url, enumerate}
\usepackage[numbers,sort&compress]{natbib}
\usepackage{amsfonts}
\usepackage{mathrsfs} 
\usepackage{appendix}
\usepackage{mathtools}
\usepackage{amsmath}
\usepackage{amssymb}
\usepackage{amsthm}
\usepackage{bbm}
\usepackage{latexsym}
\usepackage{natbib}
\usepackage[mathscr]{euscript}
\usepackage{graphicx}
\usepackage{tikz}
\usepackage{pgfplots}
\usepackage{pgfplotstable}
\usepackage{fullpage}
\usetikzlibrary{arrows,positioning,chains,fit,shapes,calc}
\newtheorem{theorem}{Theorem}[section]
\newtheorem{lemma}[theorem]{Lemma}

\newtheorem{proposition}[theorem]{Proposition}

\newtheorem{example}[theorem]{Example}

\newtheorem{conjecture}[theorem]{Conjecture}

\newtheorem{definition}[theorem]{Definition}

\newtheorem{question}[theorem]{Question}

\pgfplotsset{compat=1.8}

\def\NN{{\mathbb N}}

\DeclareMathOperator{\diam}{diam}

\DeclareMathOperator*{\ext}{extremal}

\def\epsilon{\varepsilon}
\let\eps=\varepsilon

\renewcommand{\dots}{\ldots}

\title{Maximum and minimum degree conditions for embedding trees}
\date{\today}

\author{Guido Besomi, Mat\'ias Pavez-Sign\'e\footnote{MPS was supported by ANID	Doctoral scholarship ANID-PFCHA/Doctorado Nacional/2017-21171132.}, and Maya Stein\footnote{MS is also affiliated to Centro de Modelamiento Matem\'atico, Universidad de Chile, UMI 2807 CNRS. MS acknowledges support by CONICYT + PIA/Apoyo a centros cient\'ificos y tecnol\'ogicos de excelencia con financiamiento Basal, C\'odigo AFB170001 and by Fondecyt Regular Grant 1183080.}\\ \ \\
Departamento de Ingenier\'ia Matem\'atica\\ Universidad de Chile\\  Beauchef 851, Santiago, Chile\\
\texttt{\{gbesomi, mpavez, mstein\}@dim.uchile.cl}
}

\begin{document}

\maketitle

\begin{abstract}
We propose the following conjecture: For every fixed $\alpha\in [0,\frac 13)$, each graph of minimum degree at least $(1+\alpha)\frac k2$ and maximum degree at least $2(1-\alpha)k$ contains each tree with $k$ edges as a subgraph.  
Our main result is an approximate version of the conjecture  for bounded degree trees and large dense host graphs. We also show that our conjecture is asymptotically best possible.\\
The proof of the approximate result relies on  a second result, 
  which might be  interesting on its own. Namely, we can embed any bounded degree tree into host graphs of minimum/maximum degree asymptotically exceeding $\frac k2$ and $\frac 43k$, respectively, as long as the host graph avoids a specific structure. 
\end{abstract}

\section{Introduction}

A central challenge in extremal graph theory is to determine degree conditions for subgraph containment. The aim is to find bounds on the average/median/minimum degree of a graph~$G$ which ensure  that $G$ contains all graphs of a fixed class $\mathcal H$ as subgraphs. One of the most interesting open cases is when $\mathcal H$ is the class of all trees of some fixed size $k\in\mathbb{N}$.  Let us give a quick outline of the most relevant directions that have been suggested in the literature.

\paragraph{Minimum degree.}
It is very easy to see that every graph of {\em minimum} degree at least $k$ contains each tree with $k$ edges, and this is sharp (consider the disjoint union of complete graphs of order $k$). 

\paragraph{Average degree.}
The famous Erd\H os--S\'os conjecture from 1964 (see~\cite{Erdos64}) states that every graph with {\em average} degree strictly greater than $k-1$ contains each tree with $k$ edges. If true, the conjecture is sharp. 

This conjecture has received a lot of attention over the last three decades, in particular, a proof  was announced by Ajtai, Koml\'os, Simonovits and Szemer\'edi in the early 1990's. Many particular cases have been settled since then, see e.g.~\cite{brandt96,sacle97,rohzon,BPS3,hax,goerlich2016}. 

\paragraph{Median degree.}
The Loebl--Koml\'os--S\'os conjecture from 1992 (see~\cite{Loebl95}) states that every graph of {\em median} degree at least $k$ contains each tree with $k$ edges.  If true, also this conjecture is sharp. For the  case $k=\frac n2$, Ajtai, Koml\'os and Szemer\'edi~\cite{AKS} proved an approximate version for large $n$, and years later Zhao~\cite{Zhao2011} proved the exact result for large $n$. 

An approximate version of the Loebl--Koml\'os--S\'os conjecture for dense graphs was proved by Piguet and Stein~\cite{PS12}. The exact version for dense graphs was settled by Piguet and Hladk\'y~\cite{LKSdense}, and independently by Cooley~\cite{Cooley:2009}. For sparse graphs, Hladk\'y, Koml\'os, Piguet, Szemer\'edi and Stein proved an approximate version of the Loebl--Koml\'os--S\'os conjecture in a series of four papers~\cite{LKS1, LKS2, LKS3, LKS4}.

\paragraph{Maximum and minimum degree.}

A new angle to the tree containment problem was introduced in 2016 by  Havet, Reed, Stein, and Wood~\cite{2k3:2016}. They impose bounds on both the minimum {\em and} the maximum degree. More precisely, they suggest that every graph of minimum degree at least $\lfloor{\frac{2k}{3}}\rfloor$ and maximum degree at least~$k$ contains each tree with $k$ edges. Again, this is sharp if true. We call this conjecture the {\em $\frac 23$--conjecture}, for progress see~\cite{2k3:2016, RS18a, RS18b}. 

In~\cite{BPS1}, the present authors  proposed a variation of this approach, conjecturing that every graph of  minimum degree at least $\tfrac{k}{2}$ and maximum degree at least $2k$ contains each tree with~$k$ edges. We call this  the {\em  $2k$--$\frac k2$ conjecture}. An example illustrating the sharpness of the conjecture, and a version  for trees with maximum degree bounded by $k^{\frac1{67}}$ and large dense host graph can be found in~\cite{BPS1}. 

\paragraph{New conjecture.}

Comparing the two variants of maximum/minimum degree conditions given by the previous two conjectures, it seems natural to ask whether one can allow for a wider spectrum of bounds for the maximum and the minimum degree of the host graph.  We believe that it is possible to weaken the bound on the maximum degree given by the  $2k$--$\frac k2$ conjecture, if simultaneously, the bound on the minimum degree is increased. Quantitatively speaking, we suggest the following.

\begin{conjecture}
\label{conj:ell}
Let $k\in\mathbb N$, let $\alpha\in[0,\frac 13)$ and let $G$ be a graph with $\delta(G)\geq  (1+\alpha)\frac k2$ and $\Delta (G)\geq 2(1-\alpha)k$. Then $G$ contains each tree with $k$ edges. 
\end{conjecture}

Note that for $\alpha= 0$,  the bounds  from Conjecture~\ref{conj:ell} conincide with the bounds from the $2k$--$\frac k2$ conjecture. In contrast, the case $\alpha=\frac 13$ is not included in Conjecture~\ref{conj:ell} as we believe that the appropiate value for the maximum degree is $k$ and not
 $\frac{4k}3$ if the minimum degree is $\frac 23k$ (as suggested by the $\frac 23$-conjecture).

We show that Conjecture~\ref{conj:ell} is asymptotically best possible for infinitely many values of~$\alpha$.

\begin{proposition}\label{prop:ell'}
For all odd $\ell\in\mathbb N$ with $\ell\ge 3$, and for all $\gamma>0$ there are $k\in \mathbb N$, a $k$-edge tree~$T$, and a graph $G$ with $\delta(G)\geq  (1+\frac 1\ell -\gamma)\frac k2$ and $\Delta (G)\geq 2(1-\frac 1\ell-\gamma)k$ such that $T$ does not embed in $G$. 
\end{proposition}
We prove Proposition~\ref{prop:ell'} in Section~\ref{tightness}. Note that Proposition~\ref{prop:ell'} covers all values of $\alpha\in\{ \frac 15, \frac 17, \frac 19, \ldots\}$.  The tightness of our conjecture for other values of $\alpha$ will be discussed in Section~\ref{tight?}. We remark that Proposition~\ref{prop:ell'} disproves a conjecture
 from~\cite{rohzon} (see Section~\ref{tightness} for details).

 On the positive side, observe that Conjecture~\ref{conj:ell} trivially holds for stars and for  double stars. Also, it is not difficult to see that the conjecture  holds for paths. In fact, if the tree we wish to embed is the $k$-edge path $P_k$, it already suffices to require a minimum degree of  at least $\frac k2$ and a maximum degree of at least  $k$ in the host graph $G$.\footnote{This is enough because the latter condition forces  a component of size at least $k+1$, and thus the statement reduces to a well known result of Erd\H os and Gallai~\cite{ErdGall59}.}

We provide further evidence for the correctness of Conjecture~\ref{conj:ell} by  proving an approximate version  for bounded degree trees and large dense host graphs. 

\begin{theorem}\label{main} For all $\delta \in(0,1)$ there exist $k_0\in \NN$ such that for all $n,k\ge k_0$ with $n\ge k\ge \delta n$ and for each $\alpha\in [0,\frac 13)$ the following holds.\\ Every $n$-vertex graph $G$ with 
 $\delta(G)\geq(1+\delta)(1+\alpha)\frac{k}{2}$ and 
 $\Delta(G)\geq 2(1+\delta)(1-\alpha)k$
contains each
 $k$-edge tree $T$ with $\Delta(T)\le k^{\frac{1}{67}}$ as a subgraph.
\end{theorem}

The proof of Theorem~\ref{main} will be given in Section~\ref{sec:mainproof}. It relies on another result, namely Theorem~\ref{main:2-2} below, which we already prove in Section~\ref{4/3}, making use of a powerful embedding tool from~\cite{BPS1} (Lemma~\ref{lem:superlemma}).  We believe Theorem~\ref{main:2-2} might be interesting in its own right. 

Theorem~\ref{main:2-2} is a variant of Theorem~\ref{main}, but with the much weaker conditions
$\delta(G)\geq(1+\delta)\frac{k}{2}$ and 
 $\Delta(G)\geq(1+\delta)\frac 43k$.
Because of Proposition~\ref{prop:ell'}, these bounds are not sufficient to guarantee an embedding of any given tree $T$, but if we are not able to embed~$T$, then some  information about the structure of~$G$ can be deduced. Indeed, one can prove that, after deleting few edges and vertices, $G$ looks like the extremal graph from Proposition~\ref{prop:ell'}. Roughly speaking, an extremal graph $H$ is a graph with a vertex $x$ of high degree which sees only two components of $H-x$. The larger of such components is bipartite and $x$ sees only the larger bipartite class which, moreover, has size at least $\frac23k$. The other component may be bipartite, in which case $x$ sees only one bipartition class, or non-bipartite and with size at most $\frac 23k$.   We will give an explicit description of the corresponding class of graphs, which we will call {\em $(\eps,x)$-extremal graphs}, in Definition~\ref{defii} in Section~\ref{4/3}, but already state our result here.

\begin{theorem}\label{main:2-2} 
For all $\delta \in(0,1)$ there is $n_0\in \NN$  such that for all $k, n\ge n_0$  with $n\ge k\ge \delta n$, the following holds for every $n$-vertex graph $G$ with $\delta(G)\geq (1+\delta)\frac{k}{2}$ and $\Delta (G)\geq (1+{\delta})\tfrac{4k}{3}$. \\ If $T$ is a $k$-edge tree with  $\Delta(T)\le k^{\frac{1}{67}}$, then either 
\begin{enumerate}[(a)]
	\item $T$ embeds in $G$; or 
	\item $G$ is $(\frac{\delta^4}{10^{10}},x)$-$\ext$ for every $x\in V(G)$ of degree at least $(1+\delta)\tfrac{4k}{3}$. \end{enumerate}
\end{theorem}

We discuss some further directions and open problems in Section~\ref{sec:conclusion}. More precisely, we discuss possible extensions of Theorems~\ref{main} and the influence of some additional assumptions on the host graph, such as higher connectivity, on the degree bounds from Theorem~\ref{main}. Also, we discuss the sharpness of  Conjecture~\ref{conj:ell} for those values of $\alpha$ not covered by Proposition~\ref{prop:ell'}.

\section{Sharpness of Conjecture~\ref{conj:ell}}\label{tightness}

This section is devoted to showing the asymptotical tightness of our conjecture,  for infinitely many values of $\alpha$.
 In order to prove Proposition~\ref{prop:ell'}, let us consider the following example.
 
\begin{figure}[h!]
\centering
\includegraphics[width=0.5\linewidth]{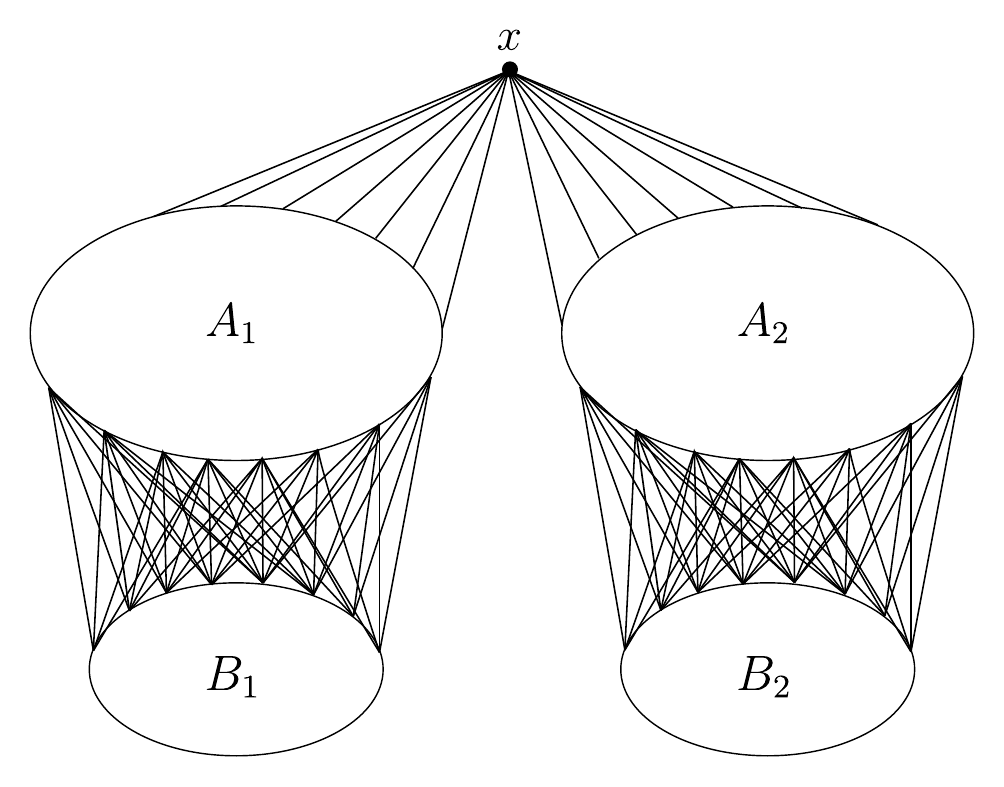}
\caption{The graph $H_{k,\ell,c}$ from Example~\ref{ex:extremal_4k/3}.}
\label{fig:Hklc}
\end{figure}

\begin{example}\label{ex:extremal_4k/3}
Let $\ell, k, c\in\NN$  with $1\le c\le \frac k{\ell(\ell +1)}$ such that $\ell\ge 3$ is odd and divides $k$. \\
For $i=1,2$, we define $H_i=(A_i, B_i)$ to be the complete bipartite graph with $$|A_i|=(\ell-1)\left(\frac{k}{\ell}-1\right)\text{\ and \ }|B_i|=\frac k2+\frac{(c-1)(\ell +1)}2-1.$$ 
We obtain $H_{k,\ell,c}$ by adding a new vertex $x$ to  $H_1\cup H_2$, and adding all edges between~$x$ and $A_1\cup A_2$. Observe that 
$$\delta (H_{k,\ell,c})=\min\{|A_1|,|B_1|+1\}=|B_1|+1=\frac k2+\frac{(c-1)(\ell +1)}2$$ 
and 
$$\Delta (H_{k,\ell,c})= |A_1\cup A_2|=2(\ell-1)\left(\frac{k}{\ell}-1\right).$$ 
\end{example}
\noindent Let us now prove Proposition~\ref{prop:ell'}.

\begin{proof}[Proof of Proposition~\ref{prop:ell'}]
Let $\ell$ and $\gamma$ be given, and consider an arbitrary $c\geq 1$. Set 
	\[k:=c\ell(\ell+1).\]
Let $T_{k,\ell}$ be the tree formed by $\ell$ stars of order $\frac{k}{\ell}$ and an additional vertex~$v$ connected to the centres of the stars. We claim that we cannot embed $T_{k,\ell}$ in $H_{k,\ell,c}$. Observe that we cannot embed $T_{k,\ell}$ in $H_{k,\ell,c}$ by mapping $v$ into $x$, since then one of the sets $B_{i}$ would have to accommodate all  leaves of at least $\frac{\ell +1}2$ of the  stars of order $\frac{k}{\ell}$. But these are at least  $$\frac{\ell +1}2\cdot \left(\frac k\ell-1\right)=\frac k2+\frac 1{2\ell}(k-\ell(\ell +1))\ge \frac k2+\frac 12(c-1)(\ell +1)> |B_i|$$ leaves in total, so they will not fit into $B_i$.

Moreover, we cannot
map $v$  into one of the $H_i$, because then we would have to embed at least $\ell-1$ stars into $H_i$. The leaves of these stars would have to go to the same side as $v$, but together these are $$(\ell -1)\left(\frac{k}{\ell} -1\right)+1>|A_i|\ge |B_i|$$ vertices (note that we count $v$), so this, too,   is impossible.
We conclude that the tree $T_{k,\ell}$ does not embed in $H_{k,\ell,c}$. Now observe that 
$$\delta (H_{k,\ell,c})>\Big(1+\frac{(c-1)(\ell +1)}k\Big)\frac k2=\Big(1+\frac{c-1}{c\ell}\Big)\frac k2=\Big(1+\frac{1}{\ell}-\frac{1}{c\ell}\Big)\frac k2$$
and
$$\Delta (H_{k,\ell,c})
=2\Big(1-\frac 1\ell-\frac{\ell-1}k\Big)k> 2\Big(1-\frac 1\ell-\frac 1{c\ell}\Big)k.$$ 
So, if we choose $c$ large enough such that
\[\delta (H_{k,\ell,c})\ge \Big(1+\frac 1\ell -\gamma\Big)\frac k2\hspace{.5cm}\text{and}\hspace{.5cm}\Delta (H_{k,\ell,c})\ge2\Big(1-\frac 1\ell-\gamma\Big)k,\]
which is as desired.
 \end{proof}
 
Note that Proposition~\ref{prop:ell'} disproves a conjecture
 from~\cite{rohzon} which suggested that any graph with maximum degree at least $\frac 43k$ and minimum degree at least $\frac k2$ contains each $k$-edge tree. Actually, a version of Proposition~\ref{prop:ell'} with $\gamma=\frac 1\ell$ (which is slightly easier to prove - just take $c=1$) would be sufficient to disprove the conjecture
 from~\cite{rohzon}.

Let us now quickly discuss an alternative example, which gives worse bounds
 than the ones given in Proposition~\ref{prop:ell'}, but might be interesting because of its different structure.

\begin{example}\label{example2}
Let $k,\ell, c$ be as in Example~\ref{ex:extremal_4k/3}. Let $C$ be a complete graph of order  $\frac k2 +\frac{(c-1)(\ell +1)}2$. Let $G_{k, \ell, c}$ be obtained by taking  $C$ and the bipartite graph  $H_1=(A_1, B_1)$ from Example~\ref{ex:extremal_4k/3}, and joining a new vertex $x$ to all vertices from $A_1$ and to all vertices in $C$.  Then $\delta(G_{k, \ell, c})=\frac k2+\frac{(c-1)(\ell +1)}2$ and $\Delta(G_{k, \ell, c})=\frac{3\ell-2}{2\ell}k+\frac{(c-3)(\ell +1)}2-2$.

Moreover, in the same way as in the proof of Proposition~\ref{prop:ell'}, we can show that if $k$ is large enough in terms of (odd) $\ell\ge 3$ and~$\gamma$, then $$\delta(G_{k, \ell, c})\ge (1+\frac 1\ell -\gamma)\frac k2\text{ \ and \ }\Delta(G_{k, \ell, c})\ge\frac 32(1-\frac 1{\ell}-\gamma)k$$
and $T_{k,\ell}$ does not embed into $G_{k, \ell, c}$. 
 \end{example}

This example, as well as the example underlying Proposition~\ref{prop:ell'}, illustrates that requiring a maximum degree of at least $ck$, for any $c<2$ (in particular for $c=\frac{4}{3}$), and a minimum degree of at least  $\frac k2$ is not enough to guarantee that any graph  obeying these conditions contains all $k$-edge tree as subgraphs. Nevertheless, we were not able to come up with any radically different examples, and it might be that graphs that look very much like the graph $H_{k,\ell,c}$ from Example~\ref{ex:extremal_4k/3} or the graph  $G_{k,\ell,c}$ from Example~\ref{example2}  are the only obstructions for embedding all $k$-edge trees. 
This suspicion is partially confirmed by  Theorem~\ref{main:2-2}.

\section{Regularity}\label{sec:regu}
 
Let $H=(A,B)$ be a bipartite graph with parts $A$ and $B$. For $X\subseteq A$ and $Y\subseteq B$, the density of the pair $(X,Y)$ is defined as  $d(X,Y):=\frac{e(X,Y)}{|X||Y|}$. Given $\varepsilon>0$, the pair $(A,B)$ is said to be {\em $\varepsilon$-regular} if \[|d(X,Y)-d(A,B)|<\varepsilon\]
for all $X\subseteq A$  and $Y\subseteq B$ such that $|X|>\varepsilon|A|$ and  $|Y|> \varepsilon |B|$.
If, moreover, $d(A,B)>\eta$ for
some fixed $\eta>0$, we call the pair  {\em $(\varepsilon,\eta)$-regular}. 

 The regularity lemma of Szemer\'edi~\cite{Sze78} states that the vertex set of any large enough graph can be partitioned into a bounded number of clusters such that almost all pairs of clusters form an $\varepsilon$-regular bipartite graph. 
We will use  the well known degree form of the regularity lemma (see for instance~\cite{regu}).\\

\noindent Call a vertex partition $V(G)=V_1\cup\dots\cup V_\ell$ an {\em $(\varepsilon,\eta)$-regular partition} if 
\begin{enumerate}
	\item $|V_1|=|V_2|=\dots=|V_\ell|$;
	\item  $V_i$ is independent for all $i\in[\ell]$; and
	\item for all $1\le i<j\le \ell$, the pair $(V_i,V_j)$ is $\varepsilon$-regular with density either $d(V_i,V_j)>\eta$ or $d(V_i,V_j)=0$.
\end{enumerate}

\begin{lemma}[Regularity lemma - Degree form]\label{reg:deg}
For all $\varepsilon>0$ and $m_0\in\NN$ there are $N_0, M_0$ such that the following holds for all $\eta\in[0,1]$ and $n\ge N_0$. Any $n$-vertex graph $G$ has a subgraph $G'$, with $|G|- |G'|\le \varepsilon n$ and $\deg_{G'}(x)\ge \deg_G(x)-(\eta+\varepsilon)n$ for all $x\in V(G')$, such that $G'$ admits an $(\varepsilon,\eta)$-regular partition $V(G')=V_1\cup\dots\cup V_\ell$, with $m_0\le \ell\le M_0$.
\end{lemma}

The {\em $(\varepsilon,\eta)$-reduced graph} $R$ of $G$, with respect to the $(\varepsilon,\eta)$-regular partition given by Lemma~\ref{reg:deg}, is the graph with vertex set $\{V_i:i\in[\ell]\}$, where $V_iV_j$ is an edge if  $d(V_i,V_j)>\eta$. We will often refer to the $(\varepsilon,\eta)$-reduced graph $R$ without explicitly referring to the associated $(\varepsilon,\eta)$-regular partition.
 A well-known fact is that $R$ inherits many  properties of $G$. For instance,  it asymptotically preserves the minimum degree of $G$ (scaled to the order of $R$). Indeed, for every $V_i$, we have
 \begin{equation}\label{R:deg}\deg_R(V_i)\ge \sum_{V_j\in N_R(V_i)}d(V_i,V_j)=\sum_{v\in V_i}\frac{\deg_{G'}(v)}{|V_i|}\cdot \frac{|R|}{|G'|},\end{equation}
 and so, in particular, one can deduce that 
 \begin{equation}\label{R:mindeg}\delta(R)\ge \delta(G')\cdot\frac{|R|}{|G'|}\ge \Big(\delta(G)-(\eta+\eps)n\Big)\cdot \frac{|R|}{|G'|}.\end{equation}

\section{Maximum degree $\frac{4k}3$}
	\label{4/3}
In this section we will  prove our tree embedding result for host graphs of maximum degree approximately  $\frac{4k}3$ and minimum degree roughly $\frac k2$, namely Theorem~\ref{main:2-2}. The proof of Theorem~\ref{main:2-2} crucially relies on an embedding result from~\cite{BPS1},  Lemma~\ref{lem:superlemma} below.  This lemma describes a series of configurations which, if they appear in a graph~$G$, allow us to embed any bounded degree tree of the right size into~$G$.

Before stating Lemma~\ref{lem:superlemma}, and defining the class of $(\eps,x)$-$\ext$ graphs (the graphs that are excluded as host graphs in  Theorem~\ref{main:2-2}), let us go through some useful notation. 

\smallskip

For a fixed $\theta\in(0,1)$, we say that a vertex $x$ of a graph~$H$ {\em $\theta$-sees} a set $U\subseteq V(H)$ if it has at least $\theta |U|$ neighbours in~$U$. If $C$ is a component of a reduced graph of $H-x$, then we say that $x$ {\em $\theta$-sees} $C$ if it has at least $\theta|\bigcup V(C)|$ neighbours in $\bigcup V(C)$.

A non-bipartite  graph $H$ is said to be {\em $(k,\theta)$-small} if $|V(H)|<(1+\theta)k$. A bipartite graph $H=(A,B)$ is said to be {\em $(k,\theta)$-small} if $$\max\{|A|,|B|\}<(1+\theta)k.$$ 
If a graph is not $(k,\theta)$-small, we will say that it is {\em $(k,\theta)$-large}.

\medskip

\noindent We are now ready to define the  excluded host graphs from Theorem~\ref{main:2-2}.

\begin{definition}[$(\eps,x)$-$\ext$]\label{defii}
Let $\eps >0$ and let $k\in\mathbb{N}$. Given a graph~$G$ and a vertex $x\in V(G)$, we say that $G$ is $(\eps,x)$-$\ext$ if for every  $(\eps,5\sqrt{\eps})$-reduced graph $R$ of $G-x$ the following  conditions hold: 
	\begin{enumerate}[(i)]
		\item every  component of  $R \text{ is } \textstyle(k\cdot\frac{|R|}{|G|},\sqrt[4]\eps)$-small;
		\item $x$ $\sqrt{\eps}$-sees two  components $C_1$ and $C_2$ of $R$ and $x$  does not see any other component of~$R$; 
			\end{enumerate}
		and furthermore, assuming that $\deg(x,\bigcup V(C_1))\ge \deg(x,\bigcup V(C_2))$, 
		\begin{enumerate}[(i)]\setcounter{enumi}{2}	
		\item   $C_1$ is bipartite and $(\tfrac{2k}{3}\cdot\frac{|R|}{|G|},\sqrt[4]\eps)$-large, with $x$ only seeing the larger side of $C_1$;
		\item  if $C_2$ is non-bipartite, then $C_2$ is $(\frac{2k}{3}\cdot\frac{|R|}{|G|},\sqrt[4]\eps)$-small, and if $C_2$ is bipartite, then $x$ sees only one side of the bipartition.  
	\end{enumerate}
	\end{definition}
	
	\begin{figure}[h!]
	\centering
	\includegraphics[width=.82\linewidth]{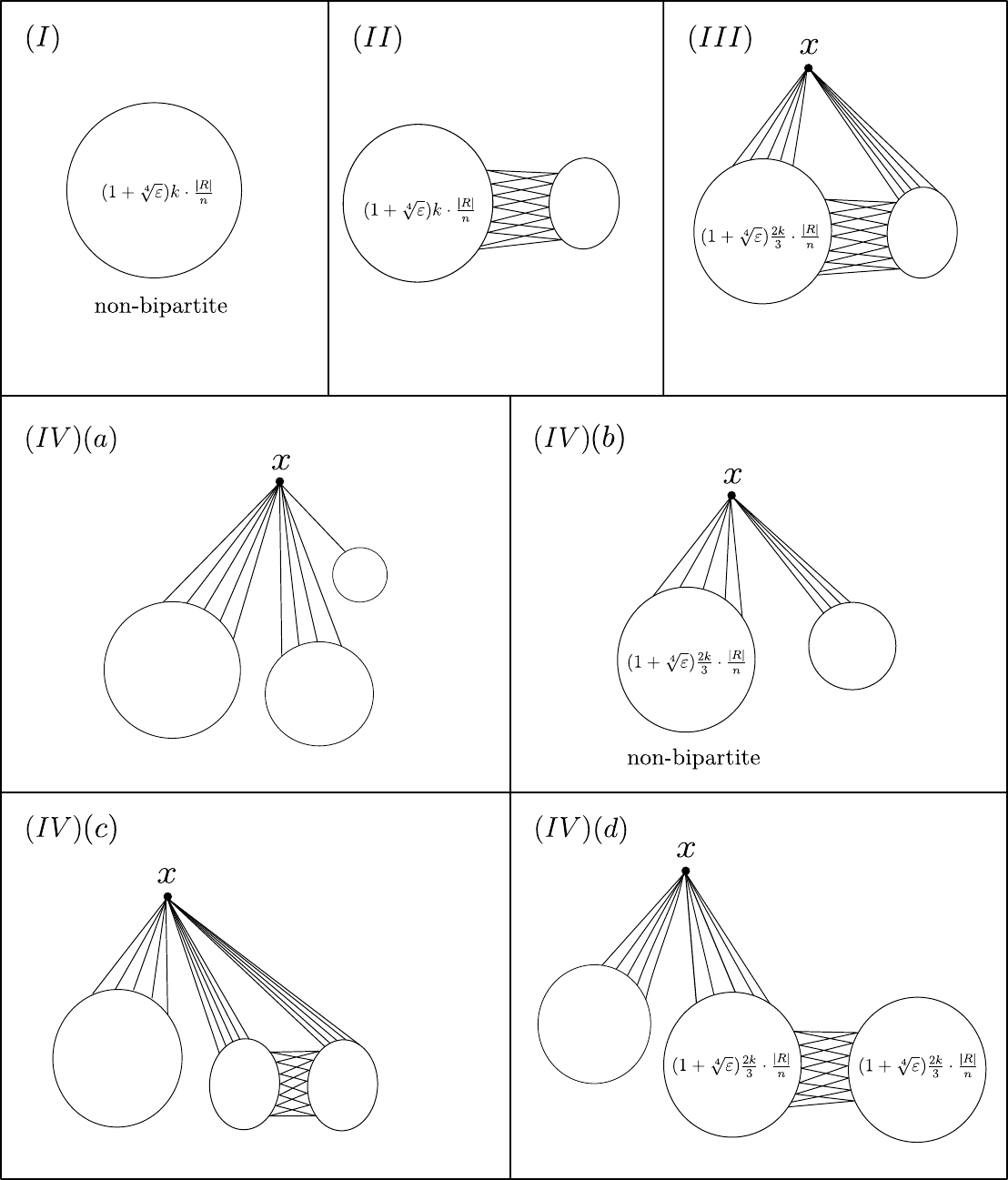}
	\caption{Cases from Lemma~\ref{lem:superlemma}.\label{alpha-cases}}
\end{figure}

\noindent We now state Lemma~\ref{lem:superlemma}. See Figure~\ref{alpha-cases} for an illustration.

\begin{lemma}$\!\!${\rm\bf\cite[Lemma 7.3]{BPS1}}\label{lem:superlemma}
	For every $\eps \in(0,10^{-10})$ and  $M_0\in\mathbb{N}$ there is $n_0\in\NN$ such that for all $n,k\ge n_0$ the following holds for every $n$-vertex graph $G$ of minimum degree at least $(1+\sqrt[4]\eps)\frac k2$. \\ 
	Let $x \in V(G)$, and suppose that $G-x$ has an $(\eps,5\sqrt\eps)$-reduced graph $R$, with $|R|\le M_0$, such that at least one of the following conditions holds:
	\begin{enumerate}[(I)]
		\item\label{cond:1} $R$ has a $(k\cdot\tfrac{|R|}{n},\sqrt[4]\eps)$-large non-bipartite component; or
		\item\label{cond:2} $R$ has a $(k\cdot\tfrac{|R|}{n},\sqrt[4]\eps)$-large bipartite component; or
		\item\label{cond:3} $R$ has a $(\tfrac{2k}{3}\cdot \tfrac{|R|}{n},\sqrt[4]\eps)$-large bipartite component such that $x$ $\sqrt\eps$-sees both sides of the bipartition; or
		\item\label{cond:4} $x$ $\sqrt\eps$-sees two components $C_1$,  $C_2$ of $R$ and one of the following holds:
		\begin{enumerate}[(a)]
			\item\label{cond:4a} $x$ sends at least one edge to a third component $C_3$ of $R$; or
			\item\label{cond:4b} $C_i$ is non-bipartite and $(\tfrac{2k}{3}\cdot \tfrac{|R|}{n},\sqrt[4]\eps)$-large for some $i\in\{1,2\}$; or
			\item\label{cond:4c} $C_i=(A,B)$ is bipartite for some $i\in\{1,2\}$, and $x$ sees both $A$ and~$B$; or
			\item\label{cond:4d}  $C_i=(A,B)$ is bipartite  for some $i\in\{1,2\}$, with $\min\{|A|,|B|\}\ge(1+\sqrt[4]\eps)\tfrac{2k}{3}\cdot \tfrac{|R|}{n}$ and $x$ seeing only one side of the bipartition.
		\end{enumerate}
	\end{enumerate}	
	Then every  $k$-edge tree $T$ of maximum degree at most $k^{\frac{1}{67}}$ embeds in $G$.
\end{lemma}

Let us remark that the original result in~\cite{BPS1}  covers even more cases than we chose to reproduce here.  

We also remark that the bound $k^{\frac 1{67}}$ on the maximum degree of the trees comes from the embedding strategy used in the first two cases of Lemma~\ref{lem:superlemma} which is as follows. We divide the tree $T$ into a constant number of tiny subtrees $T_1, T_2, \dots$ and embed these following a DFS order, embedding $T_1$ anywhere. Having embedded $\bigcup_{j<i} T_j$, we look for a regular pair $(C_i,D_i)$ having enough free space to embed $T_i$. As  we are working in a large connected component of the reduced graph, there is  a path $P_i$ (in the reduced graph) from the cluster containing the already embedded ancestor of $T_i$ to $C_i$. We embed the upper levels of $T_i$ along $P_i$, and the bulk of $T_i$ into the pair $(C_i,D_i)$. Note that if  $P_i$ has length $\ell$,  we use at most $\Delta(T)^{\ell}$ vertices from clusters on $P_i$. This is negligible if $\Delta(T)^{\ell}=o(k)$, and so, assuming that the size of the reduced graph is bounded by $M_0$, we could work with trees $T$ obeying  $\Delta(T)\le k^{\frac 1{O(M_0)}}$. In order to relax this bound to $k^{\frac 1{67}}$, extra efforts have to be made   in order to use only short paths $P_i$. Any improvement in this argument would lead to an improvement in our bound on the maximum degree of the trees.   

We are now ready for the proof of Theorem~\ref{main:2-2}.

\begin{proof}[Proof of Theorem~\ref{main:2-2}]
Given $\delta\in(0,1)$, we set 
\begin{equation}\label{epsi}
\eps:=\frac{\delta^4}{10^{10}}.
\end{equation}

 Let $N_0,M_0$ be given by~Lemma~\ref{reg:deg}, with input $\varepsilon$, $\eta:=5\sqrt{\eps}$ and $m_0:=\frac{1}{\eps}$, and let $n'_0$ be given by Lemma~\ref{lem:superlemma}, with input $\varepsilon$ and $M_0$. We choose $$n_0:=(1-\eps)^{-1}\max\{n'_0,N_0\}+1$$ as the numerical output of Theorem~\ref{main:2-2}.
 
Let $G$ and $T$ be given as in Theorem~\ref{main:2-2}. Consider an arbitrary vertex $x\in V(G)$ with $\deg(x)\ge (1+\delta)\frac43k$, and
 apply Lemma~\ref{reg:deg} to $G-x$. We obtain a subgraph $G'\subseteq G-x$ which admits an $(\varepsilon,5\sqrt{\varepsilon})$-regular partition of $G-x$, with corresponding  $(\varepsilon,5\sqrt\eps)$-reduced graph $R$.  Note that  $$\delta(G')\geq (1+\tfrac{\delta}2)\frac{k}{2}\ge (1+100\sqrt[4]\eps)\frac{k}{2}.$$  
If $R$ has a $(k\cdot \tfrac{|R|}{|G'|},\sqrt[4]\eps)$-large component, we are either  in scenario~\eqref{cond:1} or~\eqref{cond:2}  from Lemma~\ref{lem:superlemma}, and we can embed $T$. So let us assume this is not the case. In particular, we can assume that condition (i) of Definition~\ref{defii} holds.

Since $G'$ misses less than $\varepsilon n+1$ vertices from $G$, we have that
\begin{equation}\label{maxim}
\deg_{G}(x,G')\ge (1+\tfrac{\delta}{2})\frac 43k\ge(1+100\sqrt[4]\eps)\frac43k.
\end{equation}
It is clear that $x$ has to $\sqrt{\eps}$-see at least one component $C_1$ of $R$. Indeed, otherwise, we would have that
 \begin{equation}\label{eq:1comp}
 	\frac 43\delta n\le \frac 43k\le\deg_{G}(x,G')=\sum_{C} \textstyle\deg_{G}(x,\bigcup V(C))\le\sqrt\eps n,
 \end{equation}
 where the sum is over all components $C$ of $R$, and this contradicts~\eqref{epsi}.
 Suppose that $x$ sees only one component. Since $C_1$ is $(k\cdot \tfrac{|R|}{|G'|},\sqrt[4]\eps)$-small and $\deg_G(x,\bigcup V(C))\ge(1+\frac{\delta}{2})\frac{4k}{3}$, it follows that $C_1$ is bipartite and thence the largest bipartition class of $C_1$ has size at least $(1+\tfrac\delta2)\frac{2k}3\cdot \frac{|R|}{|G'|}$ and $x$ $\sqrt\eps$-sees both bipartition classes. Therefore we are in scenario~\eqref{cond:3} from Lemma~\ref{lem:superlemma} and thus we can embed $T$. 
 
 Suppose from now that $x$ sends edges outside of $C_1$. Since $C_1$ is $(k\cdot \tfrac{|R|}{|G'|},\sqrt[4]\eps)$-small, it follows that 
 \begin{equation}\label{eq:2comp}\textstyle\deg_G(x,G'\setminus \bigcup V(C_1))\ge (1+50\sqrt[4]\eps)\dfrac{k}{3}.\end{equation}
We claim that $x$ $\sqrt{\eps}$-sees at least two components of $R$. Indeed, since $k\ge \delta n$ and from~\eqref{eq:2comp} we have
\[\frac{\delta n}3\le (1+50\sqrt[4]\eps)\frac{k}{3}\le \sum_{C\not=C_1}\deg_G(x,\bigcup V(C))\le \sqrt\eps n,\]
which contradicts~\eqref{epsi}. 

If $x$ sends at least one edge to a third  component, then we are in scenario~\eqref{cond:4a} from Lemma~\ref{lem:superlemma} and thus $T$ can be embedded. Therefore, we know that $x$ actually  $\sqrt\eps$-sees exactly two components, which we will call $C_1$ and $C_2$ (In particular, we know that condition (ii) of Definition~\ref{defii} holds). By symmetry, we may assume that $\deg(x,\bigcup V(C_1))\ge\deg(x,\bigcup V(C_2))$ and thus,  by~\eqref{maxim},
\begin{equation}\label{equii}
\textstyle\deg(x,\bigcup V(C_1))\ge (1+100\sqrt[4]\eps)\dfrac{2k}{3}.
\end{equation}
 Thus, if $C_1$ is non-bipartite we are in scenario~\eqref{cond:4b} from Lemma~\ref{lem:superlemma}, and therefore, we can assume $C_1=(A_1,B_1)$ is bipartite. Also,  $x$ only sees one side of the bipartition, say $A_1$, since otherwise we are in scenario~\eqref{cond:4c}. Moreover,  by~\eqref{equii}, and since we may assume we are not in scenario~\eqref{cond:4d}, we know that
\begin{equation}\label{size:C1C2}|A_1|\ge (1+100\sqrt[4]\eps)\frac{2k}{3}\cdot\frac{|R|}{|G'|}\hspace{.5cm}\text{and}\hspace{.5cm}|B_1|\le (1+\sqrt[4]\eps)\frac{2k}{3}\cdot\frac{|R|}{|G'|}.\end{equation}
So, condition (iii) of Definition~\ref{defii} holds. 

Furthermore, if $C_2$ is non-bipartite, then it  is $(\tfrac{2k}{3}\cdot\tfrac{|R|}{|G'|},\sqrt[4]\eps)$-small, as otherwise we are in case~\eqref{cond:4b}. If $C_2$ is bipartite, then $x$ can only see one side of the bipartition, since otherwise we are in scenario~\eqref{cond:4c}. Therefore, $C_2$ satisfies condition (iv) of Definition~\ref{defii}, implying that $G$ is $(\frac{\delta^4}{10^{10}},x)$-$\ext$.
\end{proof}

\section{The proof of Theorem~\ref{main}}\label{sec:mainproof}

This section contains the proof of our main result, Theorem~\ref{main}. We first need to collect some preliminary results. 

Our first lemma is folklore, it states that every tree $T$ has a cutvertex which separates $T$ into  subtrees of size at most $\frac{|T|}2$. The proof is straightforward and can be found  for instance in~\cite{2k3:2016}.

\begin{lemma}\label{lem:cut1}
Every tree $T$ with $t$ edges has a vertex $z$ such that every component of $T-z$ has at most  $\lceil\frac t2\rceil$ vertices.
\end{lemma}
A vertex $z$ as in  Lemma~\ref{lem:cut1} is called a  {$\frac t2$-separator} for $T$.
We also need the following lemma from~\cite{BPS1}, which will allow us to conveniently  group the components of $T-z$ obtained from Lemma~\ref{lem:cut1} when is applied to a tree $T$ .
\begin{lemma}$\!\!${\rm\cite[Lemma 4.4]{BPS1}}\label{lem:num}
Let $m,t \in \NN_+$ and let $(a_i)_{i=1}^m$ be a sequence of  integers  with $0<a_i \leq \lceil\frac t2\rceil$, for each $i\in [m]$, such that $\sum_{i=1}^m a_i \leq t$. Then there is a partition $\{J_1,J_2\}$ of $[m]$ such that  $\sum_{i\in J_2} a_i \leq \sum_{i\in J_1} a_i \leq \frac{2}{3}t$.

\end{lemma}
Finally, we need another embedding result from~\cite{BPS1}. This result will enable us to embed any bounded degree forest into any large enough bipartite graph with an underlying $(\eps,\eta)$-regular partition of a certain structure. \\

\noindent Let us first define the kind of forest we are interested in.

\begin{definition}
Let $t_1, t_2\in\mathbb{N}$ and let $c\in(0,1)$. We say that a forest $F$, with bipartition classes $C_1$ and $C_2$, is a $(t_1, t_2, c)$-forest if 
\begin{enumerate}
	\item $|C_i|\le t_i$ for $i=1,2$; and
	\item  $\Delta(F)\leq (t_1+t_2)^{c}$.\end{enumerate}
	\end{definition}
\noindent We are now ready for the embedding result.
	
\begin{lemma}$\!\!${\rm\cite[Corollary 5.4]{BPS1}}\label{emb:forest}
	For all $\varepsilon\in (0,10^{-8})$ and  $d,M_0\in\NN$  there is $t_0$ such that for all $n, t_1, t_2\geq t_0$ 
	the following holds. 
	Let $G$ be a $n$-vertex graph  having an $(\varepsilon,5\sqrt\varepsilon)$-reduced graph $R$, with $|R|\le M_0$, such that
	\begin{enumerate}[(i)]
		\item $R=(X,Y)$ is connected and bipartite;
		\item $\diam(R)\leq d$;
		\item $\deg(x)\ge (1+100\sqrt\varepsilon)t_2 \cdot \frac{|R|}n$, for all $x\in X$; and
		\item $|X|\ge (1+100\sqrt\varepsilon)t_1 \cdot \frac{|R|}n$.
	\end{enumerate}
Then any $(t_1,t_2,\frac 1d)$-forest $F$, with colour classes $C_0$ and $C_1$, can be embedded into $G$, with $C_0$ going to $\bigcup X$ and $C_1$ going to $\bigcup Y$. \\
	Moreover, if $F$ has at most $\tfrac{\eps n}{|R|}$ roots, then the roots going to $\bigcup X$ can be mapped to any prescribed set of size at least $2\varepsilon|\bigcup X|$ in $X$, and the roots going to $\bigcup Y$ can be mapped to any prescribed set of size at least $2\varepsilon|\bigcup Y|$ in~$Y$.
\end{lemma}
\noindent We are now ready for the proof of our main theorem,  Theorem~\ref{main}.

\begin{proof}[Proof of Theorem~\ref{main}]
Given $\delta\in (0,1)$, we set $$\eps:=\frac{\delta^4}{10^{10}}$$ and 
apply Lemma~\ref{reg:deg}, with inputs $\eps$, $\eta=5\sqrt\eps$ and  $m_0:=\frac 1\eps$, to obtain numbers $n_0$ and~$M_0$. Next, apply Lemma~\ref{emb:forest},  with input $\eps$ and  further inputs $d:=3$ and~$M_0$ to obtain a number $k'_0$.
Choose $k_0$ as the larger of $n_0$, $k'_0$ and the output of Theorem~\ref{main:2-2}.

Now, let $k,n\in\mathbb N$, let $\alpha\in [0,\frac 13)$, let $T$ be a tree and let $G$ be a graph as in Theorem~\ref{main}. Let $x$ be an arbitrary vertex of maximum degree in $G$. Note that 
\begin{equation*}
\deg_G (x)\geq 2(1+{\delta})(1-\alpha)k\geq (1+{\delta})\tfrac{4k}{3}.
\end{equation*}
We apply the regularity lemma (Lemma~\ref{reg:deg}) to $G-x$ to obtain a subgraph $G'\subseteq G-x$ which admits an $(\eps, 5\sqrt\eps)$-regular partition with a corresponding reduced graph $R$. Moreover, since $G'$ misses only few vertices from $G$, we know that
\begin{equation}\label{deg_x}
\deg_G(x,G')\ge 2(1+\tfrac\delta2)(1-\alpha)k\hspace{.5cm}
\end{equation}
 and
 \begin{equation}\label{deg_R}
 \delta(G')\ge (1+\tfrac{\delta}{2})(1+\alpha)\frac{k}{2},
\end{equation}
and thus
 \begin{equation}\label{mindegR}
\delta (R)\geq (1+\tfrac\delta 2)(1+\alpha)\frac{k}{2}\cdot\frac{|R|}{|G'|}.
\end{equation}

Apply Theorem~\ref{main:2-2} to $T$ and  $G$. This either yields an embedding of~$T$, which would be as desired, or tells us that $G$ is an $(\eps,x)$-$\ext$ graph. We assume the latter from now on. \\

\noindent So, we know that $x$ $\sqrt\eps$-sees two components~$C_1$ and~$C_2$ of $R$, where $C_1=(A, B)$ is bipartite, say with $|A|\ge |B|$. Moreover,  $x$ does not see any other component of~$R$. Furthermore, 
\begin{enumerate}[(A)]
\item $C_i$ is $(k\cdot\frac{|R|}{|G'|},\sqrt[4]\eps)$-small, for $i=1,2$; and\label{AAA}
		\item $C_1$ is $(\tfrac{2k}{3}\cdot\frac{|R|}{|G'|},\sqrt[4]\eps)$-large, and $x$ does not see $B$.\label{CCC111}
\end{enumerate}
By~\eqref{deg_x}, and since we assume that $x$ sends more edges to $\bigcup V(C_1)$ than to $\bigcup V(C_2)$, we know that 
\begin{equation}\label{deg_x_C1}
\textstyle\deg_G(x,\bigcup V(C_1)	)\ge (1+\tfrac{\delta}2)(1-\alpha)k,
\end{equation}
and thus, by~\eqref{CCC111},
\begin{equation}\label{size_of_A}
|C_1|\ge |A|\ge (1+\tfrac{\delta}2)(1-\alpha)k\cdot\frac{|R|}{|G'|},
\end{equation}
since $x$ has at least that many neighbours in $A$, because of inequality~\eqref{deg_x_C1}.

Also,
note that  because of~\eqref{AAA} and because of the bound~\eqref{mindegR}, we know that any pair of clusters from the same bipartition class of $C_1$ has a common neighbour. Therefore, 
\begin{equation}\label{diameter}
\text{the diameter of $C_1$ is bounded by $3$.}
\end{equation}
Let us now turn to the tree $T$. We apply Lemma~\ref{lem:cut1} to find  a $\frac t2$-separator~$z$ of $T$. Let $\mathcal F$ denote the set of all components of $T-z$. Then 
\begin{equation}\label{compoF}
\text{each component of $\mathcal F$ has size at most $\Big\lceil\frac t2\Big\rceil$.}
\end{equation}

 \begin{figure}[h!]
 	\centering
 	\includegraphics[width=.9\linewidth]{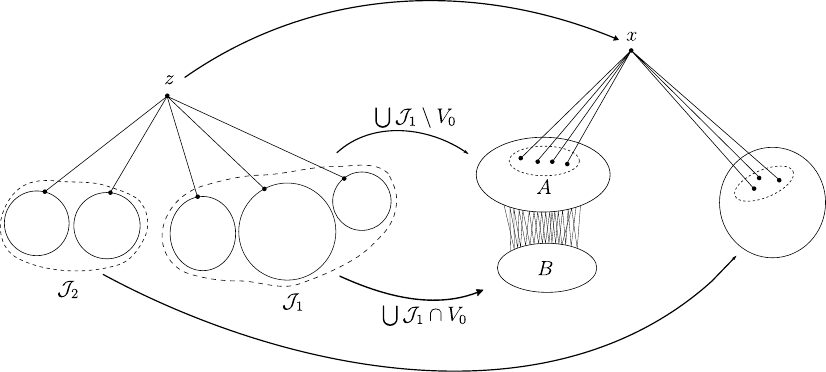}
 	\caption{Embedding if~\eqref{size_of_V0} does not to hold.}
 	\label{fig:embedding1}
 \end{figure}
 
Let $V_0$ denote the set of all vertices of $T-z$ that lie at even distance to $z$.
 We claim that if we cannot embed $T$, then
\begin{equation}\label{size_of_V0}
|V_0|\ge (1+\alpha)\frac k2.
\end{equation}
Indeed, suppose otherwise. Then we can apply Lemma~\ref{lem:num} to obtain a partition of $\mathcal F$ into two sets $\mathcal J_1$ and $\mathcal J_2$ such that
\[ \textstyle|\bigcup \mathcal J_1|\le \dfrac 23k\text{ \ \ and \ \ }|\bigcup \mathcal J_2|\le \dfrac k2.\]
We embed $z$ into $x$. Our plan is to 
 use Lemma~\ref{emb:forest} with reduced host graph~$C_1$, and with \[t_1+t_2:=\textstyle|\bigcup\mathcal J_1|\le\dfrac 23k\]
where $t_1:=|\bigcup\mathcal J_1\setminus V_0| $ and $t_2:=|\bigcup\mathcal J_1\cap V_0|$ are the sizes of the two bipartition classes of $\bigcup\mathcal  J_1$. Since we assumed~\eqref{size_of_V0} does not to hold, we have
\begin{equation}\label{upper:V0}t_2\le|V_0|\le(1+\alpha)\frac k2.\end{equation}
We now embed $\bigcup \mathcal J_1$ into~$C_1$, with the roots of $\mathcal J_1$ embedded in the neighbourhood of $x$. Observe that  condition (iii) of Lemma~\ref{emb:forest} holds because of~\eqref{mindegR} and~\eqref{upper:V0}, and condition (iv) holds because of~\eqref{size_of_A}. Moreover, the neighbourhood of $x$ is large enough to accommodate the roots of the trees from $\mathcal J_1$ because of~\eqref{deg_x_C1} and the bound on $\Delta(T)$. In order to see condition~(ii) of  Lemma~\ref{emb:forest}, it suffices to recall~\eqref{diameter}.

 Also, because of~\eqref{deg_R}, and since  $x$ also $\sqrt{\eps}$-sees the  component  $C_2$, we can embed the trees from~$\mathcal J_2$  into $C_2$. We do this by first mapping the roots of the trees from~$\mathcal J_2$ into the neighbourhood of $x$ in $C_2$. Then, since the minimum degree of $G'$ is larger than $|\bigcup\mathcal J_2|$ we may complete the embedding of $\bigcup\mathcal J_2$ greedily. In this way, we have embedded all of~$T$, as desired.

 So, from now we can and will assume that~\eqref{size_of_V0} holds. We split the remainder of the proof into two complementary cases, which will be solved in  different ways. Our two cases are defined according to whether or not there is a tree $F^*\in\mathcal F$ such that $|V(F^*)\cap V_0|>\alpha k$. Let us first treat the case where such a tree $F^*$ does not exist.

\begin{figure}[h!]
	\centering
	\includegraphics[width=.9\linewidth]{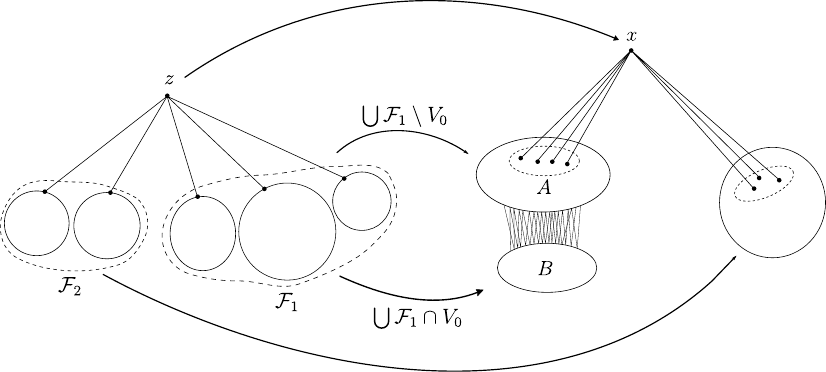}
	\caption{Embedding in Case 1.}
	\label{fig:embedding3}
\end{figure}

\bigskip

\noindent{\bf Case 1:} $|V(F)\cap V_0|\le\alpha k$ for each $F\in \mathcal F$.

\medskip

In this case, we proceed as follows. First, we embed $z$ into $x$.
We take an inclusion-maximal subset $\mathcal F_1$ of $\mathcal F$ such that
\begin{equation}\label{F1upper}
\textstyle|\bigcup \mathcal F_1\cap V_0|\le (1+\alpha)\dfrac k2
\end{equation}
holds. Then, because of the maximality of $\mathcal F_1$ and our assumption on $|V(F)\cap V_0|$ for the trees $F\in\mathcal F$, we know that 
\begin{equation}\label{F1lower}
\textstyle|\bigcup \mathcal F_1\cap V_0|\ge (1-\alpha)\dfrac k2.
\end{equation}
Hence, the trees from $\mathcal F_1$ can be embedded into $C_1$, by using Lemma~\ref{emb:forest} as before, with $t_1+t_2:=|\bigcup\mathcal F_1|$ where $t_1:=|\bigcup \mathcal F_1\setminus V_0|$ and $t_2:=|\bigcup \mathcal F_1\cap V_0|$. Indeed, inequalities~\eqref{F1upper} and~\eqref{mindegR} ensure that condition (iii) of the lemma holds.
Furthermore, because of~\eqref{size_of_A} and~\eqref{F1lower}, we know that \[t_1=\textstyle|\bigcup \mathcal F_1\setminus V_0|\le (1+\alpha)\dfrac k2\le\dfrac{1}{1+\tfrac{\delta}{2}} |\bigcup V(A)|,\]
and hence, it is clear that also condition (iv) of Lemma~\ref{emb:forest} holds. 

Condition~(ii) of  Lemma~\ref{emb:forest} holds because of~\eqref{diameter}.
Finally, inequality~\eqref{deg_x_C1} ensures we can  embed $\mathcal F_1$ in $C_1$ in such a way the roots of $\mathcal F_1$ are embedded into the neighbourhood of $x$ in $C_1$.

Now, the trees from $\mathcal F_2:=\mathcal F\setminus \mathcal F_1$ can be embedded into $C_2$. First, embed the neighbours of $z$ into the neighbourhood of $x$ in $C_2$. Then, observe that~\eqref{F1lower} implies that 
$$\textstyle|\bigcup \mathcal F_2|\le (1+\alpha)\dfrac k2\le \delta(G').$$
Therefore, we can  embed the remainder of the trees from $\mathcal F_2$ into $C_2$ in a greedy fashion.

\begin{figure}[h!]
	\centering
	\includegraphics[width=.9\linewidth]{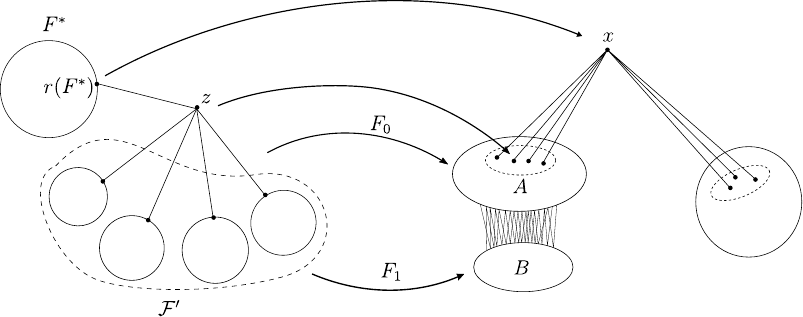}
	\caption{Embedding in Case 2.}
	\label{fig:embedding2}
\end{figure}

\bigskip

\noindent{\bf Case 2:} There is a tree $F^*\in\mathcal F$ such that $|V(F^*)\cap V_0|>\alpha k$.\\

\noindent In this case, let us set $\mathcal F':=\mathcal F\setminus\{F^*\}$ and note that
\begin{equation}\label{F'}
\textstyle|\bigcup \mathcal F'\cap V_0|\le (1-\alpha)k.
\end{equation}
Our plan is to embed $z$ into a neighbour of $x$ in $A$, and embed all trees from~$\mathcal F'$ into $C_1$. We then complete the embedding by mapping the root of~$F^*$ to $x$, and the rest of $F^*$  to $C_2$.

For the embedding of $\{z\}\cup \bigcup\mathcal F'$, we will  use Lemma~\ref{emb:forest} as before, but this time the roles of $A$ and $B$ will be reversed. That is, all of $$F_0:=\textstyle(\{z\}\cup \bigcup\mathcal F')\cap V_0$$ is destined to go to $A$, while all of $$\textstyle F_1:=(\{z\}\cup \bigcup\mathcal F')\setminus V_0$$ is destined to go to $B$. \\

We choose  $t_1+t_2:=|\bigcup\mathcal F'|+1$ where $t_1$ and $t_2$ are the sizes of the bipartition classes of $\{z\}\cup \bigcup\mathcal F'$, that is, we set $t_1:=|F_0|$ and $t_2:=|F_1|$.
Because of~\eqref{size_of_V0}, there are at most $(1-\alpha)\frac k2$ vertices in $T-z$ lying at odd distance from~$z$. In particular, $t_2\le (1-\alpha)\frac k2$. So,  by~\eqref{mindegR}, we know that  condition (iii) of Lemma~\ref{emb:forest} holds (and condition~(i) is obviously true).

Now, condition (iv) of Lemma~\ref{emb:forest} is ensured by inequality~\eqref{F'} together with~\eqref{size_of_A}. Observe that condition~(ii) of  Lemma~\ref{emb:forest} holds because of~\eqref{diameter}.
Therefore, we can embed all of $\{z\}\cup \bigcup\mathcal F'$ with the help of Lemma~\ref{emb:forest}. Furthermore, we can make sure that $z$ is embedded into a neighbour of $x$.

It remains to embed the tree $F^*$. We embed its root $r(F^*)$ into $x$, and embed all the neighbours of $r(F^*)$ into arbitrary neighbours of $x$ in $C_2$. We then embed the rest of $F^*$ greedily inyo $C_2$. Note that this is possible, since by~\eqref{compoF}, we know that
\[|F^*-r(F^*)|\le\Big\lceil\frac k2\Big\rceil -1,\]
and so, our bound~\eqref{deg_R} guarantees that the minimum degree in $C_2$ is large enough to  embed the remainder of $F^*$ greedily into $C_2$.
\end{proof}

\section{Concluding remarks}\label{sec:conclusion}

\subsection{Lower bounds on the minimum degree}\label{lower_min}

Let us now discuss why  variants of Conjecture~\ref{conj:ell} (or of Theorem~\ref{main}) with  bounds  on the minimum degree that are below the threshold $\frac k2$ are not possible. In fact, if we do not add further restrictions and the minimum degree of the host graph $G$ is only bounded by some function $f(k)< \lfloor\frac k2\rfloor$, then no maximum degree bound can make $G$ contain all $k$-edge trees.
In order to see this, it suffices consider $K_{n_1, n_2}$, the complete bipartite  graph with classes of size $n_1:=\lfloor{\tfrac{k-1}{2}\rfloor}$ and $n_2:=n-\lfloor{\tfrac{k-1}{2}\rfloor}$, respectively. No perfectly (or almost perfectly) balanced\footnote{We say that a tree $T$ is perfectly (or almost perfectly) balanced if the bipartition classes of $T$ are of the same size (or differ by one).} $k$-edge tree  embeds into $K_{n_1, n_2}$, since one would need to use at least $\lfloor\tfrac{k+1}{2}\rfloor$ vertices from each class.

One might think that  perhaps the situation changes if we require the minimum degree bound $f(k)$  to be at least as large as the smaller bipartition class of the tree. But that is not true:  Let $k,\ell\in\mathbb N$ such that $\ell +2$ divides $k+1$, and let $T$ be obtained from a $2\frac{k+1}{\ell}$-edge path  by adding $\ell-2$ new leaf neighbours to every other vertex on the path. This tree has bipartition classes of sizes $\frac{1}\ell (k+1)$ and $\frac{\ell -1}\ell (k+1)$. However, it cannot be embedded into  the graph obtained by joining  a universal vertex to a disjoint union of (any number of) complete graphs of order $c:=\lfloor\frac {k-1}2\rfloor$, since for every $v\in V(T)$, at least one component of $T-v$ has  at least $\tfrac k2 >c$ vertices.

\subsection{Higher connectivity}

All the examples from the previous section, as well as Examples~\ref{ex:extremal_4k/3} and~\ref{example2} from Section~\ref{tightness}, have a cutvertex. So one  might think that in a $c$-connected host graph, for some $c\ge2$ which might even depend on $k$, we could cope with lower bounds on the minimum (or maximum) degree.

However, we would not gain much by requiring higher connectivity, as the following variation of Example~\ref{ex:extremal_4k/3}  
 illustrates.

\begin{example}~\label{example_matching}
Let $H_{k,\ell,c}$ be as in Example~\ref{ex:extremal_4k/3}, with slightly adjusted size of the sets $A_i$, namely, we choose
$$|A_i|=(\ell-1)\Big(\frac{k}{\ell}-2\Big)\text{\ and \ }|B_i|=\frac k2+\frac{(c-1)(\ell +1)}2-1.$$
Now, add a matching of size $|B_1|$ between the sets $B_1$ and $B_2$, and call the new graph $H'_{k,\ell,c}$. 
 \end{example}
 
 The graph $H'_{k,\ell,c}$ is $\frac k2$-connected, and
for any given $\gamma$ there is a number $c$ such that
$\delta (H'_{k,\ell,c})\ge (1+\frac 1\ell -\gamma)\frac k2$ and $\Delta (H'_{k,\ell,c})\ge 2(1-\frac 1\ell-\gamma)k$. But, 
 similarly, we can show that  the tree $T_{k, \ell}$ does not embed in  $H'_{k,\ell,c}$.\\ 

\noindent It is not clear what happens if we require a connectivity of  $(1+\eps)\frac k2$, for some $0<\eps \le\alpha$. It is possible that then, the bound on the maximum degree can be weakened, perhaps  to $\Delta(G)\geq 2(1-2\alpha)k$.

\subsection{Is Conjecture~\ref{conj:ell} tight for all values of $\alpha$?}\label{tight?}

Finally, we believe it would be very interesting to generalise Proposition~\ref{prop:ell'} to even $\ell$, if this is possible. Or even better, find examples so that  the term~$\frac 1\ell$ from the proposition can be replaced with any $\alpha\in[0,\frac 13)$.
\begin{question}\label{question}
 Is  Conjecture~\ref{conj:ell} asymptotically tight for all $\alpha\notin\{\frac 1\ell\}_{\ell\ge 3, \ell\text{\ odd}}$?
\end{question}
At least, Conjecture~\ref{conj:ell} is close to tight. Indeed, for any $\alpha\in [0,\tfrac13)$ and given $\gamma>0$ small, we can construct examples of graphs with minimum degree at least $(1+\alpha-\gamma)\tfrac k2$ and maximum degree at least $2(1-g(\alpha)-\gamma)k$, where $g(\alpha)$ is a function which is bigger than~$\alpha$ but reasonably close to it. In particular, $g(\alpha)$ satisfies $|\alpha-g(\alpha)|=O(\alpha^2)$, and, more explicitly, for any even $\ell\ge 3$   we obtain $g(\tfrac1\ell)=\tfrac1{\ell}+\tfrac{1}{\ell(\ell-2)}$. These examples  are very similar to Example~\ref{ex:extremal_4k/3}. The difference is that the small stars that make up the tree may have different sizes (more precisely, one star is smaller than the other ones). The host graph is the same, with slightly adjusted size of the sets $A_i$.

\bibliographystyle{acm}
\bibliography{trees}

\end{document}